\documentclass[11pt]{article}
\usepackage{amsmath,amsthm, amssymb, amsfonts,latexsym,amscd}
\usepackage{tikz}
\usetikzlibrary{shadings,patterns}
\newtheorem{theorem}{Theorem}[section]
\newtheorem{corollary}[theorem]{Corollary}
\newtheorem{lemma}[theorem]{Lemma}

\begin{document}
\title{\bf Wiener index of graphs with fixed number of pendant or cut vertices}
\author{Dinesh Pandey\footnote{Supported by UGC Fellowship scheme (Sr. No. 2061641145), Government of India} \and Kamal Lochan Patra}
\date{}
\maketitle

\begin{abstract}
The Wiener index of a connected graph is defined as the sum of the distances between all unordered pair of its vertices. In this paper, we characterize the graphs which extremize the Wiener index among all graphs on $n$ vertices with $k$  pendant vertices. We also characterize the graph which minimizes the Wiener index over the graphs on $n$ vertices with $s$ cut vertices.\\



\end{abstract}

\section{Introduction} 

Throughout this paper, graphs are finite, simple, connected and undirected. Let $G$ be a graph with vertex set $V(G)$ and edge set $E(G).$ For a vertex $v \in V(G),$ $ N_G(v)$ denotes the set of all neighbours of $v$ in $G.$ A vertex of degree one is called a {\it pendant vertex}. A vertex $v$ of $G$ is called a {\it cut-vertex} if $G\setminus v$ is disconnected. The distance  between two vertices  $ u,v \in V(G)$, denoted by $d_G(u,v)$ or $d(u,v)$ (if the context is clear), is the number of edges in a shortest path joining $u$ and $v$. The distance of a vertex $v \in V(G),$ denoted by $D_G(v)$, is defined as $D_G(v)= \sum_{u \in V(G)} d_G(u,v).$ We refer to \cite{We} for undefined notations and terminologies.

The {\it Wiener index} of $G$, denoted by $W(G)$, is defined as the sum of distances between all unordered pair of its vertices. $\it{i.e.}$ $$W(G)=\underset{\{u,v\}\subseteq V(G)}{\sum}d_G(u,v)=\frac{1}{2}\underset{v\in V(G)}{\sum}D_G(v).$$

Different names such as {\it graph distance} \cite{Ejs}, {\it transmission} \cite{Jp}, {\it total status} \cite{Buc} and {\it sum of all distances} \cite{G,Yg} have been used to study the graphical invariant $W(G)$. Apparently, the chemist H. Wiener was the first to point out in 1947 (see \cite{Wi}) that $W(G)$ is well correlated with certain physio-chemical properties of the organic compound from which $G$ is derived. The {\it mean distance}  \cite{D1,W1} or {\it the average distance} \cite{A,D} between the vertices is a quantity closely related to $W(G).$ By considering $G$ as an interconnection network connecting many processors, the average distance of $G$ between the nodes of the network is a measure of the average delay for traversing the messages from one node to another.

In Mathematical literature, the Wiener index is first studied by Entringer et al. in \cite{Ejs}. This gave an important direction to the researchers to characterize the graphs with extremal Wiener index in certain classes of graphs. In last $20$ years a lot of studies for the optimal graphs in different classes of trees and unicyclic graphs have been done (see \cite{Hlw,Jt,Lp,T,Twl,W,Wg,Yf,Zx}).
Apart from trees and unicyclic graphs, some other classes of graphs are also studied for the characterization of graphs having extremal Wiener index. Wiener index of graphs with fixed maximum degree is studied in \cite{S}. The graphs with maximum and minimum Wiener index among all Eulerian graphs on $n$ vertices are characterized in \cite{Gcr}.

Wiener index of unicyclic graphs with fixed number of pendent vertices or cut vertices is studied in \cite{Twl}. In this paper, we characterize the graphs having maximum and minimum Wiener index over all connected graphs on $n$ vertices with $k$ pendant vertices. We also obtain the graph which minimizes the Wiener index among all connected graphs on $n$ vertices with $s$ cut-vertices.

\subsection{Main results}

We first construct some classes of graphs. For $g < n$, let $U_{n,g}^p$ be the graph obtained by attaching $n-g$ pendant vertices at one vertex of the cycle $C_g$ and $U_{n,g}^l$ be the graph obtained by joining an edge between a pendant vertex of the path $P_{n-g}$ with a vertex of $C_g$. 

Let $ \mathfrak{H}_{n,k} $ denote the class of all connected graphs on $n$ vertices and $k$ pendant vertices. Let $\mathfrak{T}_{n,k}$ be the subclass of $ \mathfrak{H}_{n,k} $ containing all the trees on $n$ vertices and $k$ pendant vertices.  

The path $ [v_1 v_2\ldots v_n]$ on $n$ vertices is denoted by $P_n$. For positive integers $k,l,d$ with $n=k+l+d$, let $T(k,l,d)$ be the tree  obtained by taking the path  $P_d$  and adding $k$ pendant vertices adjacent to $v_1$ and $l$ pendant vertices adjacent to $v_d$. Note that $T(1,1,d)$ is a path on $d+2$ vertices. 

 We define a specific subclass of graphs in $ \mathfrak{H}_{n,0} $ as follows. Let $m_1,m_2$ and  $n$ be positive integers with $m_1,m_2\geq 3$ and $n \geq m_1+m_2-1.$ If $n > m_1+m_2-1,$ take a path on $n-(m_1+m_2)+2$ vertices and identify one pendant vertex of the path with a vertex of $C_{m_1}$ and another pendant vertex with a vertex of $C_{m_2}.$ If $n= m_1+m_2-1$, then identify one vertex of $C_{m_1}$ with a vertex of $C_{m_2}.$ We denote this graph by $C_{m_1,m_2}^n$. 
 
In this paper, we prove the following results:
\begin{theorem}\label{max-thm0}
Let $0\leq k \leq n-2$ and let $G \in \mathfrak {H}_{n,k}$.  Then
\begin{enumerate}
\item[(i)] for $2 \leq k \leq n-2,$ $W(G)\leq W\left(T(\lfloor\frac{k}{2}\rfloor,\lceil \frac{k}{2}\rceil,n-k)\right)$  and equality happens if and only if $G= T(\lfloor\frac{k}{2}\rfloor,\lceil \frac{k}{2}\rceil,n-k).$ Furthermore, $W\left(T(\left\lfloor\frac{k}{2}\right\rfloor,\left\lceil \frac{k}{2}\right\rceil,n-k)\right)=$

$$\begin{cases}{n-k+1 \choose 3}+\frac{k^2}{4}(n-k+3)+\frac{k}{2}[(n-k)^2+n-k-2] &\mbox{if k is even} \\ 
{n-k+1 \choose 3}+\frac{k^2-1}{4}(n-k+3)+\frac{k}{2}[(n-k)^2+n-k-2]+1 & \mbox{if k is odd.}\end{cases}$$

\item[(ii)] for $k=1,$ $W(G) \leq W(U_{n,3}^l)$  and equality holds if and only if $G=U_{n,3}^l.$ Furthermore, $$W(U_{n,3}^l)=\frac{n^3-7n+12}{6}.$$

\item[(iii)] for $k=0$ and $n\geq 7$, $W(G)\leq W(C_{3,3}^n)$ and equality holds if and only if $G=C_{3,3}^n.$ Furthermore, $$W(C_{3,3}^n)=\frac{n^3-13n+24}{6}.$$
\end{enumerate}
\end{theorem}

For $ 0\leq k \leq n-3 $ and $ n \geq 4,$ let $P_n^k $ be the graph obtained by adding $k$ pendant vertices at one vertex of the complete graph $K_{n-k}.$

\begin{theorem}\label{min-thm0}
Let $0\leq k \leq n-2$ and let $G \in \mathfrak {H}_{n,k}$. Then
\begin{enumerate}
\item[(i)] for $0 \leq k \leq n-3$, $ W(P_n^k) \leq W(G) $  and equality holds if and only if $ G=P_{n}^k $. Furthermore,  $$W(P_n^k)= {n-k \choose 2}+k^2+2k(n-k-1).$$

\item[(ii)] for $k=n-2,$ $W(T(1,n-3,2)) \leq W(G)$  and equality holds if and only if $G=T(1,n-3,2)$. Furthermore, $$W(T(1,n-3,2))=n^2-n-2.$$
\end{enumerate}
\end{theorem} 
Let $T_{n,k} \in \mathfrak{T}_{n,k} $ be the tree that has a vertex $v$ of degree $k$ and $T_{n,k}\setminus v = r P_{q+1} \cup (k-r) P_q$, where $q=\lfloor\frac{n-1}{k} \rfloor$ and $r = n-1 -kq$. Here, we have $0\leq r <k.$

\begin{theorem}\label{min-thm1}
Let $2\leq k\leq n-2$ and  $T \in \mathfrak{T}_{n,k}.$ Then $W(T_{n,k}) \leq W(T)$  and equality holds if and only if $T=T_{n,k}.$ 
\end{theorem}
 Let $\mathfrak{C_{n,s}}$ be the set of all connected graphs on $n$ vertices and $s$ cut vertices. For $2\leq m\leq n,$ let $v_1,v_2,\ldots,v_m$ be the vertices of a complete graph $K_m$. For $i=1,2,\ldots,m$ consider the paths $P_{l_i}$ such that $l_1+l_2+\cdots+l_m=n$. Identify a pendant vertex of the path $P_{l_i}$ with the vertex $v_i,$ for $i=1,2,\ldots,m$ to obtained a graph on $n$ vertices and we denote it by $K_m^n(l_1,l_2,\ldots, l_m)$.

\begin{theorem}\label{min-thm2}
 Let $0\leq s\leq n-3$ and $i,j\in \{1,2,\ldots,n-s\}.$ Then  the graph $K_{n-s}^n(l_1,l_2,\ldots,l_{n-s})$ with $|l_i-l_j|\leq 1$ has the minimum Wiener index over $\mathfrak{C}_{n,s}.$ 
\end{theorem}
 In the next section we  will discuss some results related to Wiener index of graphs which are useful to prove our main theorems.

\section{Preliminaries} \label{section-2}

We start this section with the following lemma.
\begin{lemma} \label{ed}
Let $G$ be a graph and $u,v \in V(G)$ are non adjacent. Let $G'$ be the graph obtained from $G$ by joining the vertices $u$ and $v$ by an edge. Then $W(G')<W(G).$ 
\end{lemma}
 It follows from Lemma \ref{ed}  that among all connected graphs on $n$ vertices, the Wiener index is minimized by the complete graph $K_n$  and maximized by a tree. Among all trees on $n$ vertices, the Wiener index is minimized by the star $K_{1,n-1}$ and maximized by the path $P_n$ (see \cite{We}, Theorem 2.1.14).
It is easy to determine the Wiener index of the following graphs(see \cite{We}): (i)$W(K_n)={n \choose  2}$ (ii) $W(P_n)={n+1 \choose  3}$ (iii) $W(K_{1,n-1})=(n-1)^2.$ The Wiener index of the cycle $C_n$ is (see \cite{Jp},Theorem 5)
\begin{equation}\label{eq-c1}
W(C_n)=\begin{cases}  \frac{1}{8}n^3 &\textit{if n is even}\\  \frac{1}{8}n(n^2-1) & \textit{if n is odd.}\end{cases}
\end{equation}
Also for $u\in V(C_n)$

\begin{equation}\label{eq-c2}
D_{C_n}(u)=\begin{cases}\frac{n^2}{4} &\textit{if n is even} \\ 
\frac{n^2-1}{4} & \textit{if n is odd.}\end{cases}
\end{equation}
 
The following lemma is very useful.
\begin{lemma} (\cite{Bsv},Lemma 1.1)\label{count}
Let $G$ be a graph and $u$ be a cut vertex in $G$. Let $G_1$ and $G_2$ be two subgraphs of $G$ with $G= G_1 \cup G_2$ and $V(G_1) \cap V(G_2)= \{u\}$. Then
$$W(G)=W(G_1)+W(G_2)+(|V(G_1)|-1)D_{G_2}(u)+(|V(G_2)|-1)D_{G_1}(u).$$
\end{lemma}

\begin{corollary}\label{new1}
Let $G$ and $H$ be two connected graphs having at least $2$ vertices each. Let $u,v \in V(G)$ and $w \in V(H)$. Let $G_1$ and $G_2$ be the graphs obtained from $G$ and $H$ by identifying the vertex  $w$ of $H$ with the vertices $u$ and $v$ of $G$, respectively. If $D_G(v)\geq D_G(u)$ then $W(G_2)\geq W(G_1)$ and equality happens if and only if $D_G(v) = D_G(u).$
\end{corollary}
\begin{proof}
By Lemma \ref{count}, 
$$W(G_1)=W(G)+W(H)+(|V(G)|-1)D_H(w)+(|V(H)|-1)D_G(u)$$ and 
$$W(G_2)=W(G)+W(H)+(|V(G)|-1)D_H(w)+(|V(H)|-1)D_G(v).$$ So 
$$W(G_2)-W(G_1)=(|V(H)|-1)(D_G(v)-D_G(u))$$ and the result follows.
\end{proof}

Let  $G$ be a connected graph on $n\geq 2$ vertices. Let $v$ be a vertex of $G.$ For $l,k \geq 1,$  let $G_{k,l}$ be the
graph obtained from $G$ by attaching two new paths $P:vv_{1}v_{2}\cdots v_{k}$ and $Q:vu_{1}u_{2}\cdots u_{l}$ of
lengths $k$ and $l$ respectively, at $v$, where $u_{1},u_{2},\ldots,u_{l}$ and $v_{1},v_{2},\ldots,v_{k}$ are distinct new vertices. Let ${\widetilde G}_{k,l}$ be the graph obtained from $G_{k,l}$ by removing the edge $\{v_{k-1},v_{k}\}$ and adding the edge $\{u_{l},v_{k}\} $. Observe that the graph ${\widetilde G}_{k,l}$ is isomorphic to the graph
$G_{k-1,l+1}$. We say that ${\widetilde G}_{k,l}$ is obtained from $G_{k,l}$ by {\em grafting} an edge. 

Consider the path $P_n:v_1 v_2\ldots v_n$ on $n$ vertices with $v_i$ adjacent to $v_{i-1}$ and $v_{i+1}$ for $2\leq i\leq n-1.$ Then for $i=1,2,\ldots,n$, $$D_{P_n}(v_i)=D_{P_n}(v_{n-i+1})=\dfrac{(n-i)(n-i+1)+i(i-1)}{2}.$$ So, if $n$ is odd, then
$$D_{P_n}(v_1)>D_{P_n}(v_2)>\cdots>D_{P_n}(v_{\frac{n+1}{2}})<D_{P_n}(v_{\frac{n+3}{2}})<\cdots <D_{P_n}(v_{n-1})<D_{P_n}(v_n)$$
and if $n$ is even, then 
$$D_{P_n}(v_1)>D_{P_n}(v_2)>\cdots>D_{P_n}(v_{\frac{n}{2}})=D_{P_n}(v_{\frac{n+2}{2}})<\cdots<D_{P_n}(v_{n-1})<D_{P_n}(v_n).$$
The next result follows from the above and Corollary \ref{new1}.
\begin{corollary}\label{effect-2}(\cite{Lp},Lemma 2.4)
If $1\leq k \leq l,$ then $W(G_{k-1,l+1})>W(G_{k,l}).$
\end{corollary}

The following result compares the Wiener index of two graphs, where one is obtained from the other by moving one component from a vertex to another vertex.
\begin{lemma}(\cite{Ll},Lemma 2.4)\label{effect-0}
Let $ H,X,Y $ be three connected pairwise vertex disjoint graphs having at least $2$ vertices each. Suppose that $u$ and $v$ are two distinct vertices of $H$, $x$ is a vertex of $X$ and $y$ is a vertex of $Y.$ Let G be the graph obtained from $H, X, Y $ by identifying $u$ with $x$ and $v$ with $y$, respectively. Let $G_1^*$  be the graph obtained from $H, X, Y$ by identifying vertices $u,x,y$ and let $G _2^* $ be the graph obtained from $H, X, Y$ by identifying vertices $v,x,y$ (see figure \ref{fig}). Then $W(G_1^*)< W(G)$ or W $(G_2^*) < W (G)$.
\end{lemma}

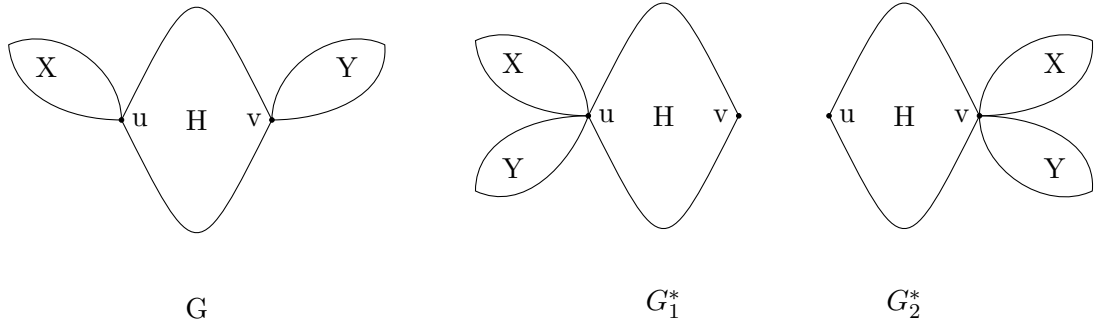
\begin{figure}[h]
\begin{center}
\begin{tikzpicture}
\filldraw (1,0)node[right]{u} circle [radius=.3mm];
\filldraw (3,0)node[left]{v} circle [radius=.3mm] (2,0)node{H} (0,.7)node{X} (4,.7)node{Y} (2,-2.5)node{G};
\draw (1,0)..controls(2,2)..(3,0)..controls(2,-2)..(1,0);
\draw (1,0) to[out=90,in=25] (-.5,1) to[out=275,in=180] (1,0);
\draw (3,0) to[out=90,in=155] (4.5,1) to[out=275,in=0] (3,0);
\end{tikzpicture}
\hskip 1cm
\begin{tikzpicture}
\filldraw (1,0)node[right]{u} circle [radius=.3mm];
\filldraw (3,0)node[left]{v} circle [radius=.3mm] (2,0)node{H} (0,.7)node{X} (0,-.7)node{Y}(2,-2.5)node{$G_1^*$};
\draw (1,0)..controls(2,2)..(3,0)..controls(2,-2)..(1,0);
\draw (1,0) to[out=90,in=25] (-.5,1) to[out=275,in=180] (1,0);
\draw (1,0) to[out=180,in=90] (-.5,-1) to[out=-25,in=250] (1,0);
\end{tikzpicture}
\hskip 1cm
\begin{tikzpicture}
\filldraw (1,0)node[right]{u} circle [radius=.3mm];
\filldraw (3,0)node[left]{v} circle [radius=.3mm] (2,0)node{H}  (4,.7)node{X} (4,-.7)node{Y} (2,-2.5)node{$G_2^*$};
\draw (1,0)..controls(2,2)..(3,0)..controls(2,-2)..(1,0);
\draw (3,0) to[out=-90,in=-155] (4.5,-1) to[out=-275,in=0] (3,0);
\draw (3,0) to[out=90,in=155] (4.5,1) to[out=275,in=0] (3,0);
\end{tikzpicture}
\end{center}
\caption{Movement of a component from one vertex to other}\label{fig}
\end{figure}

\begin{corollary}\label{effect-1}
Let $G$ be a connected graph on $n \geq 2 $ vertices and let $u,v \in V(G)$. For $n_1,n_2 \geq 0,$ let $G_{uv}(n_1,n_2)$ be the graph obtained from $G$ by attaching  $n_1$ pendant vertices at $u$ and $n_2$ pendant vertices at $v$. If $n_1,n_2 \geq 1$ then
$$W(G_{uv}(n_1+n_2,0))< W(G_{uv}(n_1,n_2))\;\;\mbox{or} \;\; W(G_{uv}(0,n_1+n_2))< W(G_{uv}(n_1,n_2)).$$
\end{corollary}
In \cite{Yf}, Lemma 2.6, if we take $G_0=P_{n_0}$ and $u_0$ and $v_0$ as two distinct pendant vertices of $G_0$, then $G_0\cong G_1 \cong G_2.$ So, $W(G_0)=W(G_1)=W(G_2),$ hence the statement of the mentioned lemma is not true. In the following result, we have given a proof of the corrected version of it.
\begin{lemma}\label{effect-3}
Let $G$ be a connected graph on $n \geq 3$ vertices and $u,v \in V(G).$ For $l,k\geq 1,$ let $G_{uv}^p(l,k)$ be the graph obtained from $G$ by identifying a pendant vertex of the path $P_l$ with $u$ and  identifying a pendant vertex of the path $P_k$ with $v$.  Suppose $l,k\geq 2.$ If $G$ is not the $u$-$v$ path and $D_G(u)\geq D_G(v)$ then $$W(G_{uv}^p(l+k-1,1)) >W(G_{uv}^p(l,k)).$$
\end{lemma}

\begin{proof}
First consider the graph $G_{u,v}^p(l,1)$ as $H$ and let $w$ be the pendant vertex of $H$ corresponding to $P_l.$ Then by Lemma \ref{count},

$$W(G_{u,v}^p(l,k))=W(H)+W(P_k)+(|V(H)|-1)D_{P_k}(v)+(k-1)D_H(v)$$
and 
$$W(G_{u,v}^p(l+k-1,1))=W(H)+W(P_k)+(|V(H)|-1)D_{P_k}(w)+(k-1)D_H(w).$$

As $D_{P_k}(v)=D_{P_k}(w)$ we get,

$$W(G_{u,v}^p(l+k-1,1))-W(G_{u,v}^p(l,k))=(k-1)(D_H(w)-D_H(v)).$$ 
Now $$D_H(w)=D_{P_{l-1}}(w)+(l-1)|V(G)|+D_G(u)$$ and $$D_H(v)=D_G(u)+(l-1)(d_G(u,v)+1)+D_{P_{l-1}}(u')$$ where $u'$ is the vertex on the path $P_l$ adjacent to $u.$  Since $D_{P_{l-1}}(w)=D_{P_{l-1}}(u')$, so $$D_H(w)-D_H(v)=(l-1)(|V(G)|-d_G(u,v)-1)+D_G(u)-D_G(v).$$
As $l\geq 2$ and $G$ is not the $u$-$v$ path, so $(l-1)(|V(G)|-d_G(u,v)-1)>0$. Hence the result follows from the given condition $D_G(u)\geq D_G(v).$
\end{proof}

The Wiener index of $U_{n,g}^p$ and $U_{n,g}^l$ are useful for our results and can be found in \cite{Yf}( see Theorem 1.1).
\begin{equation}\label{eq1}
W(U_{n,g}^p)=\begin{cases} \frac{g^3}{8}+(n-g)(\frac{g^2}{4}+n-1)  &\textit{ if g is even} \\ \frac{g(g^2-1)}{8}+(n-g)(\frac{g^2-1}{4}+n-1)  & \textit{if g is odd}\end{cases}
\end{equation}

\begin{equation}\label{eq2}
W(U_{n,g}^l)=\begin{cases} \frac{g^3}{8}+(n-g)(\frac{n^2+ng+3g-1}{6}-\frac{g^2}{12})  &\mbox{ if g is even} \\ \frac{g(g^2-1)}{8}+(n-g)(\frac{n^2+ng+3g-1}{6}-\frac{g^2}{12}-\frac{1}{4})  & \mbox{if g is odd}\end{cases}
\end{equation}

We next calculate the Wiener index of some more trees, which we need for the extremal bounds in some of our results. Let $S_{d,k}$ be the tree obtained by identifying a pendant vertex  of the path $P_d$ with the central vertex of the star $K_{1,k}$. By using Lemma \ref{count}, it is easy to see that 
\begin{equation}\label{eq3}
W(S_{d,k})= {d+1 \choose 3}+k^2+(d-1)k+\frac{d(d-1)k}{2}.
\end{equation}

 Then by Lemma \ref{count} and using the value of $W(S_{d,k})$ and $W(K_{1,l})$, we get
\begin{equation}\label{eq4}
W(T(l,k,d))={d+1 \choose 3}+l^2+k^2+\frac{(d^2+d-2)(k+l)}{2}+(d+1)kl.
\end{equation}

For $l \geq 2$ and $q \geq 1$, let $T_l^q$ be the tree on $lq+1$ vertices with $l$ pendant vertices having one vertex $v$ of degree $l$ and $T_l^q-v=lP_q$ ($l$ copies of $P_q$). Note that $T_1^q$ is the path $P_{q+1}.$ Then 
\begin{equation}\label{eq5}
 D_{T_l^q}(v)=l+2l+\cdots+ql=\frac{lq(q+1)}{2}.
\end{equation}
 Now by lemma \ref{count},
\begin{align*}
W(T_l^q)&=W(T_{l-1}^q)+W(T_1^q)+(l-1)qD_{T_1^q}(v)+qD_{T_{l-1}^q}(v)\\
&=W(T_{l-1}^q)+{q+2 \choose 3}+(l-1)q^2(q+1).
\end{align*}
Solving this recurrence relation we get,

\begin{equation}\label{eq6}
W(T_l^q)=l{q+2 \choose 3}+\frac{q^2l(q+1)(l-1)}{2}.
\end{equation}

\section{Proofs of Theorem \ref{max-thm0},Theorem \ref{min-thm0} and Theorem \ref{min-thm1}}

We first recall three known results related to Wiener index of graphs.

\begin{theorem}\label{pmax-thm1}(\cite{Shi},Theorem 4)
For $2\leq k\leq n-2,$  the tree $T(\lfloor\frac{k}{2}\rfloor,\lceil\frac{k}{2}\rceil,n-k)$ maximizes the Wiener index over $\mathfrak{T}_{n,k}.$
\end{theorem}
\begin{theorem}\label{pmax-thm3}(\cite{Yf},Corollary 1.2) 
Among all unicyclic graphs on $n > 4$ vertices, the graph $U_{n,3}^l$ has the maximum Wiener index.
\end{theorem}
 \begin{theorem} (\cite{Jp},Theorem 5) \label{plesnic}
 Let $G$ be a two connected graph with $n$ vertices then $W(G)\leq W(C_n)$ and equality holds if and only if $G=C_n.$
 \end{theorem}

We now compare the Wiener index of the graphs $C_{3,3}^n$ and $C_n.$
 \begin{lemma}\label{3cycles}
 For $n \geq 6$, $W(C_n)\leq W(C_{3,3}^n)$ and equality happens if and only if $n=6.$
\end{lemma}
\begin{proof} By (\ref{eq2}), we have
$W(U_{n,3}^l)=\frac{n^3-7n+12}{6}$. If $u$ is the pendant vertex of $U_{n,3}^l$ then $D_{U_{n,3}^l}(u)=D_{P_{n-2}}(u)+2(n-2)=\frac{(n-3)(n-2)}{2}+2n-4=\frac{n^2-n-2}{2}.$ For $n \geq 6$, let $u$ be the cut-vertex common to $C_3$ and $U_{n-2,3}^l$ of $C_{3,3}^n$. Then by Lemma \ref{count}, 
\begin{align}\label{eq7}
W(C_{3,3}^n)&=W(C_3)+W(U_{n-2,3}^l)+2D_{U_{n-2,3}^l}(u)+2(n-3) \nonumber\\
&=3+\frac{(n-2)^3-7(n-2)+12}{6}+(n-2)^2-(n-2)-2+2n-6 \nonumber\\
&=\frac{n^3-13n+24}{6}
\end{align}

By (\ref{eq-c1}) and (\ref{eq7}), we have 
$$W(C_{3,3}^n)-W(C_n)=\begin{cases} \frac{n(n^2-52)}{24}+4, & \mbox{if $n$ is even} \\ 
\frac{n(n^2-49)}{24}+4, & \mbox{if $n$ is odd.}\end{cases}$$

Hence the result follows.
\end{proof}

\begin{lemma}\label{c-cmn}
Let $m_1,m_2 \geq 3$ be two integers and let $n=m_1+m_2-1$. Then $W(C_n)>W(C_{m_1,m_2}^n).$
\end{lemma}
\begin{proof}
Let $v$ be the vertex of degree $4$ in $C_{m_1,m_2}^n$. First suppose $n$ is even. Then one of $m_1$ or  $m_2$ is odd and other is even. Without loss of generality, suppose $m_1$ is odd and $m_2$ is even. Then by Lemma \ref{count}, (\ref{eq-c1}) and (\ref{eq-c2}) we have
\begin{align*}
W(C_{m_1,m_2}^n)&=W(C_{m_1})+W(C_{m_2})+(m_2-1)D_{C_{m_1}}(v)+(m_1-1)D_{C_{m_2}}(v)\\ &= \frac{m_1^3-m_1}{8}+\frac{m_2^3}{8}+(m_2-1)\frac{m_1^2-1}{4}+(m_1-1)\frac{m_2^2}{4}\\
&=\frac{1}{8}(m_1^3+m_2^3+2m_1^2m_2+2m_1m_2^2-2m_1^2-2m_2^2-m_1-2m_2+2)
\end{align*}
and 
\begin{align*}
W(C_n)&=\frac{1}{8}(m_1+m_2-1)^3\\
&=\frac{1}{8}(m_1^3+m_2^3+3m_1^2m_2+3m_1m_2^2-3m_1^2-3m_2^2-6m_1m_2+3m_1+3m_2-1)
\end{align*}
The difference is 

\begin{align*}
W(C_n)-W(C_{m_1,m_2}^n)&=\frac{1}{8}(m_1^2m_2+m_1m_2^2-m_1^2-m_2^2-6m_1m_2+4m_1+5m_2-3)\\
&=\frac{1}{8}\left((m_2-1)m_1^2+(m_1-1)m_2^2+4m_1+5m_2-6m_1m_2-3\right)
\end{align*}
 An easy calculation gives
\begin{equation*}
W(C_n)-W(C_{m_1,m_2}^n)\begin{cases} = \frac{1}{4}m_2(m_2-2), & \mbox{if $m_1=3$}\\
             \geq\frac{1}{8}(3(m_1-m_2)^2+4m_1+5m_2-3), & \mbox{if $m_1
						\geq  5$}\\
  \end{cases}
\end{equation*}
which is greater than $0.$\\

Now suppose $n$ is odd. Then there are two possibilities.\\
\textbf{Case 1:} Both $m_1$ and $m_2$ are even.
\begin{align*}
W(C_{m_1,m_2}^n)&=W(C_{m_1})+W(C_{m_2})+(m_2-1)D_{C_{m_1}}(v)+(m_1-1)D_{C_{m_2}}(v)\\
&=\frac{m_1^3}{8}+\frac{m_2^3}{8}+(m_2-1)\frac{m_1^2}{4}+(m_1-1)\frac{m_2^2}{4}\\
&=\frac{1}{8}(m_1^3+m_2^3+2m_2m_1^2+2m_1m_2^2-2m_1^2-2m_2^2)
\end{align*}
\begin{align*}
W(C_n)&=W(C_{m_1+m_2-1})\\
&=\frac{1}{8}\left((m_1+m_2-1)^3-(m_1+m_2-1)\right)\\
&=\frac{1}{8}(m_1^3+m_2^3+3m_1^2m_2+3m_1m_2^2-3m_1^2-3m_2^2-6m_1m_2+2m_1+2m_2)
\end{align*}
The difference is
\begin{align*}
W(C_n)-W(C_{m_1,m_2}^n)&=\frac{1}{8}\left((m_1-1)m_2^2+(m_2-1)m_1^2-6m_1m_2+2m_1+2m_2\right)\\
& \geq \frac{1}{8}\left( 3(m_1-m_2)^2+2m_1+2m_2 \right) \\
&>0
\end{align*}
\textbf{Case 2:} Both $m_1$ and $m_2$ are odd.
\begin{align*}
W(C_{m_1,m_2}^n)&=\frac{m_1^3-m_1}{8}+\frac{m_2^3-m_2}{8}+(m_2-1)\frac{m_1^2-1}{4}+(m_1-1)\frac{m_2^2-1}{4}\\
&=\frac{1}{8}(m_1^3+m_2^3+2m_2m_1^2+2m_1m_2^2-2m_1^2-2m_2^2-3m_1-3m_2+4)
\end{align*}
and the difference is

$$W(C_n)-W(C_{m_1,m_2}^n)=\frac{1}{8}\left((m_1-1)m_2^2+(m_2-1)m_1^2-6m_1m_2+5m_1+5m_2-4\right).$$
An easy calculation gives
\begin{align*}
W(C_n)-W(C_{m_1,m_2}^n) \begin{cases} >\frac{1}{8}\left( 3(m_1-m_2)^2+5m_1+5m_2-4 \right), & \mbox{if $m_1,m_2 \geq 5$} \\
=\frac{1}{8}(2m_2^2-4m_2+2), &\mbox{if $m_1=3$}\\
=\frac{1}{8}(2m_1^2-4m_1+2), &\mbox{if $m_2=3$} 
 \end{cases}
\end{align*}

which is greater than $0$ and this completes the proof.
\end{proof}

\begin{lemma}\label{24}
Let $u$ be the pendant vertex  and $v$ be a non-pendant vertex of the unicyclic graph $U_{n,g}^l$. Then $D_{U_{n,g}^l}(u)>D_{U_{n,g}^l}(v)$.
\end{lemma}
\begin{proof}
Let $g$ be the vertex of degree $3$ in $U_{n,g}^l$ and let $g+1$ be the  vertex adjacent to $g$ not on the $g$-cycle of $U_{n,g}^l$. Then
\begin{equation}\label{lol}
D_{U_{n,g}^l}(u)=D_{P_{n-g+1}}(u)+(g-1)(n-g)+D_{C_g}(g).
\end{equation}

If $v$ is a vertex on the cycle $C_g$ of $U_{n,g}^l$ then

$$D_{U_{n,g}^l}(v)=D_{C_g}(v)+d(v,g)(n-g)+D_{P_{n-g+1}}(g)$$

and if $w$ is a non pendant vertex  of $U_{n,g}^l$ which is not on the cycle then

$$D_{U_{n,g}^l}(w)=D_{P_{n-g+1}}(w)+d(w,g)(g-1)+D_{C_g}(g)$$

Since $D_{P_{n-g+1}}(u)=D_{P_{n-g+1}}(g)$,$D_{P_{n-g+1}}(u)>D_{P_{n-g+1}}(w)$ and $D_{C_g}(g)=D_{C_g}(v),$ so
\begin{align*}
D_{U_{n,g}^l}(u)-D_{U_{n,g}^l}(v)&=(n-g)(g-1-d(v,g))>0\\
D_{U_{n,g}^l}(u)-D_{U_{n,g}^l}(w)& > (g-1)(n-g-d(w,g))>0.
\end{align*}
\end{proof}
The next corollary follows from Lemma \ref{24} and Corollary \ref{new1}.
\begin{corollary}\label{3uni1}
Let $G$ be a connected graph with at least two vertices and let $u\in V(G).$ Suppose $v$ is the pendant vertex of $U_{n,g}^l$ and $w$ is a non-pendant vertex of $U_{n,g}^l$. Let $G_1$
and $G_2$ be the graph obtained from $G$ and $H$ by identifying $u$ of $G$ with the vertices with $v$ and $w$ of $U_{n,g}^l$, respectively. Then $W(G_1)>W(G_2).$ 
\end{corollary}

\begin{lemma}\label{3uni2}
Let $u$ be a vertex of a connected graph $G.$ For $m\geq 4,$ let $G_1$ be the graph obtained by identifying the vertex $u$ of $G$ with the pendant vertex of $U_{m+1,m}^l$  and $G_2$ be the graph obtained by identifying the vertex $u$ with the pendant vertex of $U_{m+1,3}^l$. Then $W(G_2)>W(G_1)$.
\end{lemma}
\begin{proof}
By Lemma \ref{count}, we have

$$W(G_1)=W(G)+W(U_{m+1,m}^l)+(|V(G)|-1)D_{U_{m+1,m}^l}(u)+mD_G(u)$$ and 
$$W(G_2)=W(G)+W(U_{m+1,3}^l)+(|V(G)|-1)D_{U_{m+1,3}^l}(u)+mD_G(u).$$
By Theorem \ref{pmax-thm3}, $W(U_{m+1,3}^l)>W(U_{m+1,m}^l)$. So, the difference is
$$W(G_2)-W(G_1)>(|V(G)|-1)(D_{U_{m+1,3}^l}(u)-D_{U_{m+1,m}^l}(u)).$$

By (\ref{lol}), we have $D_{U_{m+1,3}^l}(u)=\frac{(m-1)(m+2)}{2}$ and
\begin{equation*}
D_{U_{m+1,m}^l}(u)=\begin{cases}m+\frac{m^2}{4} &\mbox{if n is even} \\ 
m+\frac{m^2-1}{4} & \mbox{if n is odd.}\end{cases}
\end{equation*}
So,
\begin{equation*}
D_{U_{m+1,3}^l}(u)-D_{U_{m+1,m}^l}(u)=\begin{cases}\frac{m^2-2m-4}{4} &\mbox{if m is even}\\
\frac{m^2-2m-3}{4} &\mbox{if m is odd} 
 \end{cases}
\end{equation*}
which is greater than $0$ and this completes the proof.
\end{proof}

\begin{corollary}\label{cor1}
Let $m_1,m_2\geq 3$ be two integers and let $m_1+m_2\leq n.$ Then $W(C_{3,3}^n)\geq W(C_{m_1,m_2}^n)$ and equality happens if and only if $m_1=m_2=3.$
\end{corollary}

\begin{proof}[{\bf Proof of Theorem \ref{max-thm0}:}]
\begin{enumerate}
\item[(i)] Let $G \in \mathfrak{H}_{n,k}.$ Construct a spanning tree $G'$ from $G$ by deleting some edges if required. Then by Lemma \ref{ed}, $W(G') \geq W(G)$. The number of pendent vertices of $G'$ is greater than or equal to $k$.  Suppose $G'$ has more than $k$ pendant vertices. Since $k\geq 2$, $G'$ has at least one vertex of degree greater than $2$
and two paths attached to it. Consider a vertex $v$ of $G'$ with $d(v)\geq 3$ and two paths $P_{l_1},P_{l_2},\; l_1\geq l_2$ attached at $v.$ Using grafting of edge operation on $G'$, we get a new tree $\tilde{G}$ with number of pendant vertices one less than the number of pendant vertices of $G'$ and by Corollary \ref{effect-2}, $W(\tilde{G})>W(G').$ Continue this process till we get a tree with $k$ pendant vertices from $\tilde{G}.$ By Lemma \ref{effect-2}, every step in this process the Wiener index will increase. So, we will reach at a tree of order $n$ with $k$ pendant vertices. Hence the result follows from Theorem \ref{pmax-thm1}. Then replacing $d, l$ and $k$ by $n-k,\lfloor \frac{k}{2} \rfloor$ and $\lceil \frac{k}{2} \rceil$, respectively in (\ref{eq4}), we get $W\left(T(\lfloor\frac{k}{2}\rfloor,\lceil \frac{k}{2}\rceil,n-k)\right)$.

\item[(ii)] Let $G\in \mathfrak{H}_{n,1}.$ Since $G$ is connected and has exactly one pendent vertex, it must contain a cycle. Let $C_g$ be a cycle in $G.$ If $G$ has more than one cycle, then construct a new graph $G'$ from $G$ by deleting edges from all cycles other than $C_g$ so that the graph remains connected. Then by Lemma \ref{ed}, $W(G')> W(G)$ and $G'$ is a unicyclic graph on $n$ vertices with girth $g.$ By Theorem \ref{pmax-thm3}, $W(U_{n,3}^l)\geq W(G')$ and equality happens if and only if $G' = U_{n,3}^l.$ As $U_{n,3}^l\in \mathfrak{H}_{n,1}$, so the result follows and  we get the value of $W(U_{n,3}^l)$ from (\ref{eq2}).

\item[(iii)]  Let $n\geq 7$ and let $G\in \mathfrak{H}_{n,0}$. Then we have two cases:\\

\textbf{Case 1:} For some integers $m_1,m_2 \geq 3$ with $n=m_1+m_2-1$ and $C_{m_1,m_2}^n$ is a subgraph of $G$.

Since $C_{m_1,m_2}^n$ is a subgraph of $G,$ by deleting some edges from $G$ we get $C_{m_1,m_2}^n \in \mathfrak{H}_{n,0}$ and by Lemma \ref{ed} $W(G)<W(C_{m_1,m_2}^n).$ Again by Lemma \ref{c-cmn}, $W(C_{m_1,m_2}^n)<W(C_n).$ Now the result follows from Lemma \ref{3cycles}.\\

\textbf{Case 2:} There is no integers $m_1,m_2 \geq 3$ with $n=m_1+m_2-1$ such that $C_{m_1,m_2}^n$ is a subgraph of $G$.

If $G$ is a two connected graph then by Theorem \ref{plesnic}, $W(G)\leq W(C_n)$ and the result follows from Lemma \ref{3cycles}. So let $G$ has at least one cut vertex. 

 \underline{Claim:}  $W(G) \leq W(C_{g_1,g_2}^n)$ for some $g_1,g_2 \geq 3$ and the equality holds if and only if $G=C_{g_1,g_2}^n.$

Since $G$ has a cut-vertex and no pendant vertices, so $G$ contains two cycles with at most one common vertex. Let $C_{g_1}$ and $C_{g_2}$ be two cycles of $G$ with at most one common vertex. Since $C_{m_1,m_2}^n$ with $m_1+m_2-1=n$ is  not a subgraph of $G$, so $g_1+g_2\leq n.$ Clearly $G$ has at least $n+1$ edges.

If $G$ has exactly $n+1$ edges, then there is no common vertex between $C_{g_1}$ and $C_{g_2}$ and $G=C_{g_1,g_2}^n.$ So, let $G$ has at least $n+2$ edges. Suppose $|E(G)|=n+k$, where $k\geq 2.$ Choose $k-1$ edges $\{e_1,\ldots,e_{k-1}\}\subset E(G)$ such that $e_i\notin E(C_{g_1})\cup E(C_{g_2}),\;\; i=1,\ldots, k-1$ and $G\setminus \{e_1,\ldots,e_{k-1}\}$ is connected. Let $G_1=G\setminus \{e_1,\ldots,e_{k-1}\}$ ($G_1$ may have some pendant vertices). Then by Lemma \ref{ed}, $W(G_1)> W(G).$ If $G_1$ has no pendant vertices then $G_1=C_{g_1,g_2}^n.$

Let $G_1$ has some pendant vertices. Then for some $l<n,$ $C_{g_1,g_2}^l$ is a subgraph of $G_1.$ By grafting of edges operation(if required), we can form a new graph $G_2$ from $G_1$ where $G_2$ is a connected graph on $n$ vertices obtained by attaching some paths to some vertices of $C_{g_1,g_2}^l.$ Then by Corollary \ref{effect-2}, $W(G_2)> W(G_1).$ If more than one paths are attached to different vertices of $C_{g_1,g_2}^l$ in $G_2$, then using the graph operation as mentioned in Lemma \ref{effect-3}, form a new graph $G_3$ from $G_2$, where $G_3$ has exactly one path attached to $C_{g_1,g_2}^l.$ Then by Lemma \ref{effect-3}, $W(G_3)> W(G_2).$ Let the path attached to the vertex $u$ in $C_{g_1,g_2}^l$ of $G_3.$ Then again we have two cases:\\
\textbf{Case-i:} $u \in V(C_{g_1})\cup V(C_{g_2})$\\
Without loss of generality, assume that $u\in V(C_{g_1}).$ Then the induced subgraph of $G_3$ containing the vertices of $C_{g_1}$ and the vertices of the path attached to it,  is the graph $U_{k,g_1}^l$ for some $k>g_1.$ Let $v$ be the pendant vertex of $U_{k,g_1}^l$. Since the two cycles $C_{g_1}$ and $C_{g_2}$ have at most one vertex in common, so we have two subcases:

\underline{Subcase-1}: $V(C_{g_1})\cap V(C_{g_2})=\{w\}$ 

Let $H_1$ be the induced subgraph of $G_3$ containing the vertices $\{V(G_3)\setminus V(U_{k,g_1}^l)\}\cup \{w\}.$ Clearly $H_1$ is the cycle $C_{g_2}.$ Then identify the vertex $v$ of $U_{k,g_1}^l$ with the vertex $w$ of $H_1$ to form a new graph $G_4.$  By Corollary \ref{3uni1}, $W(G_4)> W(G_3)$ and $G_4$ is the graph $C_{g_1,g_2}^n.$

\underline{Subcase-2}: $V(C_{g_1})\cap V(C_{g_2})=\phi$

Let $H_2$ be the induced subgraph of $G_3$ containing the vertices $V(G_3)\setminus V(U_{k,g_1}^l).$ In $G_3$ exactly one vertex $w_1\in U_{k,g_1}^l$ adjacent to exactly one vertex $w_2$ of $H_2.$ Form a new graph $G_5$ from $G_3$ by deleting the edge $\{w_1,w_2\}$ and adding the edge $\{v,w_2\}.$ By Corollary \ref{3uni1}, $W(G_5)> W(G_3)$ and $G_5$ is the graph $C_{g_1,g_2}^n.$\\
\textbf{Case-ii:} $u \notin V(C_{g_1})\cup V(C_{g_2})$\\
Let $w$ be the pendant vertex of $G_3$ and let $w_3$  be a vertex in $C_{g_1,g_2}^l$ of $G_3$ adjacent to $u.$ Form a new graph $G_6$ from $G_3$ by deleting the edge $\{u,w_3\}$ and adding the edge $\{w,w_3\}.$ By Corollary \ref{3uni1}, $W(G_6)> W(G_3)$ and $G_6$ is the graph $C_{g_1,g_2}^n.$ This proves our claim. 

Now from Corollary \ref{cor1}, it follows that $W(G) \leq W(C_{3,3}^n)$ and by (\ref{eq7}) $W(C_{3,3}^n)=\frac{n^3-13n+24}{6}.$ This completes the proof.
\end{enumerate}
\end{proof}
\noindent It can be checked easily that for $n\leq 5,$ the cycle $C_n$ has the maximum Wiener index over $\mathfrak{H}_{n,0}$ and for $n=6,$ the Wiener index is maximized by both the graphs $C_6$ and $C_{3,3}^n$.\\

\begin{proof}[{\bf Proof of Theorem \ref{min-thm0}:}]
\begin{enumerate}
\item[(i)] Let  $G\in \mathfrak{H}_{n,k}$ and let $v_1,v_2,\ldots,v_{n-k}$ be the non-pendant vertices of $G.$ If the induced subgraph $G [v_1,v_2,\ldots,v_{n-k}]$ is not complete, then form a new graph $G'$ from $G$ by joining all the non-adjacent non-pedant vertices of $G$ with new edges. Then $G' \in \mathfrak{H}_{n,k}$ and by Lemma \ref{ed} $W(G') < W(G).$ If $G'= P_n ^k$ then we are done, otherwise $G'$ has at least two vertices of degree greater than or equal to $n-k.$  Form a new graph $G''$ from $G'$ by moving all the pendant vertices to one of the vertex $v_1,v_2,\ldots,v_{n-k}$. Then  $G'' = P_n ^k$ and by Corollary \ref{effect-1}, the result follows. Let $u\in V(P_n ^k)$ be a vertex of degree $n-1.$ Then by Lemma \ref{count}, we have
\begin{align*}
W(P_n^k)&=W(K_{n-k})+W(K_{1,k})+(|V(K_{n-k})|-1)k+kD_{K_{n-k}}(u)\\
&= {n-k\choose 2}+k^2+2k(n-k-1).
\end{align*}

\item[(ii)] Let $G \in \mathfrak{H}_{n,n-2}.$ Then $G$ is isomorphic to a tree $T(k,l,2)$ for some $k, l\geq 1.$ If $k$ and $l$ both greater than or equal to $2$ then form the tree $T(1,n-3,2)$ from $G$ by moving pendant vertices from one end to other. The by Corollary \ref{effect-1},  $W(T(1,n-3,2))<W(G)$ and by taking $d=2,l=1$ and $k=n-3$ in (\ref{eq4}), we have
$W(T(1,n-3,2))=n^2-n-2.$ 
\end{enumerate}
\end{proof}

\begin{proof}[{\bf Proof of Theorem \ref{min-thm1}}]
We first prove that for $k \geq 3$, if $T \in \mathfrak{T}_{n,k}$ has minimum Wiener index then there is a unique vertex $v \in V(T)$ with $d(v)\geq 3.$
Let there be two vertices $u,v \in V(T)$ with $d(u)=n_1 \geq 3$, $d(v)=n_2  \geq 3.$ Let $N_T(u)=\{u_1,u_2,\ldots,u_{n_1}\}$ and $N_T(v)=\{v_1,v_2,\ldots,v_{n_2}\}$ where $u_1$ and $v_1$ lie on the path joining $u$ and $v$ ($u_1$ may be $v$ and $v_1$ may be $u$). Let $T_1$ be the largest subtree of $T$ consisting of $u,u_2,u_3,\ldots,u_{n_1-1}$ but not $u_1,u_{n_1}$ and $T_2$ be the largest subtree of $T$ containing $v,v_2,v_3,\ldots,v_{n_2-1}$ but not $v_1,v_{n_2}.$ We rename the vertices $u \in V(T_1)$ and $v \in V(T_2)$ by $u'$ and $v'$, respectively. Let $H=T \setminus \{u_2,u_3,\ldots,u_{n_1-1},v_2,v_3,\ldots,v_{n_2-1}\}.$ Construct two trees $T'$ and $T''$ from $H$, $T_1$ and $T_2$ by identifying the vertices $u,u',v'$ and $v,u',v'$, respectively. Clearly both $T', T''\in\mathfrak{T}_{n,k}$ and by Lemma \ref{effect-0}, either $W(T') < W(T)$ or $W(T'') < W(T)$ which is a contradiction.

Let $T$ be the tree which minimizes the Wiener index in $\mathfrak{T}_{n,k}$. For $k=2$, the only possible tree is the path $P_n$ which is isomorphic to $T_{n,2}.$ So assume $3 \leq k \leq n-2.$ Then there exists a unique vertex $v \in V(T)$ with $d(v) \geq 3.$ Hence the result follows from Corollary \ref{effect-2}.
\end{proof}

For $r=0,$ the tree $T_{n,k}$ is isomorphic to the tree $T_k^q$ and hence by (\ref{eq6}),
$$W(T_{n,k})=k{q+2 \choose 3}+\frac{q^2(q+1)k(k-1)}{2}.$$ For $1\leq r <k,$ by Lemma \ref{count}, we have 

$$W(T_{n,k})=W(T_r^{q+1})+W(T_{k-r}^q)+r(q+1)D_{T_{k-r}^q}(v)+(k-r)qD_{T_r^{q+1}}(v),$$ where $v$ is the vertex of $T_{n,k}$ with $T_{n,k}\setminus v = r P_{q+1} \cup (k-r) P_q$. Thus by using (\ref{eq5}) and (\ref{eq6}) the value of $W(T_{n,k})$ can be obtained.

\section{Proof of Theorem \ref{min-thm2}}

Any graph on $n$ vertices has at most $n-2$ cut vertices. The path $P_n$ is the only graph on $n$ vertices with $n-2$ cut vertices. Hence for $\mathfrak{C_{n,s}}$, we consider $0\leq s\leq n-3$. Let $\mathfrak{C_{n,s}^t}$ be the set of all trees on $n$ vertices with $s$ cut vertices. In a tree every vertex is either a pendant vertex or a cut vertex. So, $\mathfrak{C_{n,s}^t}= \mathfrak{T}_{n,n-s}.$ Hence the next result follows from Theorem \ref{pmax-thm1} and Theorem \ref{min-thm1}.
\begin{theorem}
For $0\leq s\leq n-3$, the tree $T(\lfloor\frac{n-s}{2}\rfloor,\lceil \frac{n-s}{2}\rceil,s)$ maximizes the Wiener index and the tree $T_{n,n-s}$ minimizes the Wiener index over $\mathfrak{C_{n,s}^t}.$   
\end{theorem}

 A block in a graph $G$ is a maximal connected component without any cut vertices in it. Let $B_G$ be the graph corresponding to $G$ with $V(B_G)$ as the set of blocks of $G$ and two vertices $u$ and $v$ of $B_G$ are adjacent whenever the  corresponding blocks contains a common cut vertex of $G.$ A vertex of $G$ with minimum eccentricity is called a central vertex.  We call a block $B$ in $G$, a pendant block if there is exactly one cut vertex of $G$ in $B$. The block corresponding to a central vertex in $B_G$ is called a central block of $G.$  
\begin{lemma}\label{c}
Let $G$ be a graph which minimizes the Wiener index over $\mathfrak{C_{n,s}}$. Then every block of $G$ is a complete graph.
\end{lemma}
\begin{proof}
Let $B$ be a block of $G$ which is not complete. Then there are at least two non adjacent vertices in $B.$ Let $u$ and $v$ be two non adjacent vertices in $B$. Form a new graph $G'$ from $G$ by joining the edge $\{u,v\}.$ Clearly $G'\in \mathfrak{C_{n,s}}$ and by Lemma \ref{ed} $W(G')<W(G)$, which is a contradiction.
\end{proof} 
\begin{lemma} \label{c_0}
Let $G$ be a graph which minimizes the Wiener index over $\mathfrak{C_{n,s}}.$ Then every cut vertex of $G$ is shared by exactly two blocks. 
\end{lemma}
\begin{proof}
Let  $c$ be a cut vertex in $G$ shared by more than two blocks say $B_1,B_2,\ldots,B_k, k\geq 3.$ Construct a new graph $G'$ from $G$ by joining all the non adjacent vertices of $\cup B_i, i=2,3,\ldots,k$. Then $G'\in \mathfrak{C_{n,s}}$ and by Lemma \ref{ed}, $W(G')<W(G)$ which is a contradiction.
\end{proof}
\begin{lemma}\label{c_2}
Let $m\geq 3.$ For  $i,j\in \{1,2,\ldots,m\}$, if $l_i \leq l_j-2,$ then $$W(K_m^n(l_1,\ldots,l_i+1,\ldots,l_j-1,\ldots,l_m)) < W(K_m^n(l_1,\ldots,l_i,\ldots,l_j,\ldots,l_m)).$$ 
\end{lemma}
\begin{proof}
Let $u$  be the pendant vertex of $K_m^n(l_1,\ldots,l_i+1,\ldots,l_j-1,\ldots,l_m)$ on the path $P_{l_i+1}$ and $v$ be the pendant vertex of $K_m^n(l_1,\ldots,l_i,\ldots,l_j,\ldots,l_m)$ on the path $P_{l_j}$. Let $w_1$ and $w_2$ be the vertices adjacent to $u$ and $v$, respectively. Then using Lemma \ref{count} we have 
\begin{align*}
&W(K_m^n(l_1,\ldots,l_i+1,\ldots,l_j-1,\ldots,l_m)) - W(K_m^n(l_1,\ldots,l_i,\ldots,l_j,\ldots,l_m))\\
&=D_{K_m^{n-1}(l_1,\ldots,l_i,\ldots,l_j-1,\ldots,l_m)}(w_1) - D_{K_m^{n-1}(l_1,\ldots,l_i,\ldots,l_j-1,\ldots,l_m)}(w_2). 
\end{align*}
 Since $l_i <l_j-1$ and $m\geq 3,$ so the result follows.
\end{proof}

Let $G$ be a graph in which every cut vertex is shared by exactly two blocks. Then $B_G$ is a tree. So, $B_G$ has either one central vertex or two adjacent central vertices and hence $G$ has either one central block or two central blocks with a common cut vertex. 
\begin{lemma}\label{c_3}
Let $G$ be a graph which minimizes the Wiener index over $\mathfrak{C_{n,s}}$. If $s\geq 2$, then every pendant block of $G$ is $K_2.$ 
\end{lemma}

\begin{proof}
All the blocks in $G$ are complete by Lemma \ref{c}. Suppose $B$ is a pendant block of $G$ which is not $K_2.$ Let $V(B)=\{v_1,v_2,\ldots,v_m\}$ with  $m>2.$ Assume $v_1$ is the cut vertex of $G$ in $B$ which is shared by another block $B'$ with $V(B')=\{v_1=u_1,u_2,\ldots,u_r\}$ and  $r\geq2.$ Construct a new graph $G'$ from $G$ as follows: Delete the edges $\{v_2,v_j\}, j=3,4,\ldots,m$ and add the edges $\{v_j,u_i\},j=3,4,\ldots,m$ and $i=2,3,\ldots r.$ When $G$ changes to $G'$ the only type of distances which increase are $d(v_2,v_j), j=3,4,\ldots, m.$ Each such distance increases by one and hence the total increment in distances  for $v_j,\;\; j=\{3,\ldots,m\}$ is exactly $m-2$. The distance $d(v_j,u_i), j=3,4,\ldots, m\;\;i =2,3,\ldots r$ decreases by one. Since $r\geq 2,$ the total distance decreases by such pair of vertices is at least $m-2.$ Since $s\geq 2$ there exists a vertex $w$ belonging to some other block $B''$ such that $d(v_j,w), j=3,4,\ldots m$ decreases by one. So $W(G')<W(G)$, which is a contradiction.
\end{proof}

Let $G$ be a graph in which every block is complete and every cut vertex is shared by exactly two blocks. Let $B$ be a central block in $G$. Let $B_1$ be a non central non pendant block of $G$ and $c_1,c_2 \in V(B_1)$ be two cut vertices of $G$. Suppose that the vertex $c_1$ is identified by a pendant vertex of a path $P_l$ and $c_2$ is shared by another block $B_2$ such that the vertices corresponding to $B_1, B_2$ and $B$ in the tree $B_G$ lie on a path. Let $V(B_1)=\{c_1=u_1,u_2,\ldots,u_{m_1}=c_2\}$ and $V(B_2)=\{v_1,v_2,\ldots,v_{m_2}=c_2\}$. Construct a new graph $G'$ from $G$ as follow: Delete the edges $\{c_1,u_i\}$ for all $u_i\in V(B_1)\setminus \{c_1,c_2\}$ and add the edges $\{u_i,v_j\}$ for all $u_i \in V(B_1)\setminus\{c_1,c_2\}$ and $ v_j \in V(B_2)\setminus \{c_2\}.$
\begin{lemma}\label{c_1}
Let $G$ and $G'$ be the graphs defined as above.Then $W(G')<W(G).$
\end{lemma}
\begin{proof}
For $i=2,\ldots,m_1-1$, let $H_i$ be the maximal connected component of $G$ containing exactly one vertex $u_i$ of $B_1$. Let $P_l:t_1t_2\cdots t_l$ be the path with $t_1$ identified with $c_1.$ When $G$ changes to $G'$, the only type of distances which increase in $G'$ are $d_{G'}(u,t_j)$ where $u\in \cup_{i=2}^{m_1-1}V(H_i)$ and $j=1,2,\ldots,l$. Each such distance increases by one in $G'.$ For any other pair of vertices, the distance between them either decreases or remains the same. Since $B_1$ is not a central block, for each $t_j, j=1,2,\ldots,l$ there exists a vertex $t_j' \in V(G)\setminus \left(\cup_{i=2}^{m_1-1}V(H_i) \cup \{t_1,t_2,\ldots,t_l,v_1,v_2,\ldots,v_{m_2}\}\right)$ such that $d_{G'}(u,t_j')$ decreases by one where $u\in \cup_{i=2}^{m_1-1}V(H_i)$. So, the increment in distance by the pairs $u,t_j$ are neutralized by the pairs $u,t_j'.$ Apart from this at least the distances $d_{G'}(u_i,v_j)$ for $i=2,3,\ldots,m_1-1$ and $j=1,2,\ldots,m_2-1$ decreases by one. So $W(G')<W(G).$ 
\end{proof}

\begin{proof}[{\bf Proof of Theorem \ref{min-thm2}:}]
Let $G$ be a graph which minimizes the Wiener index over $\mathfrak{C_{n,s}}.$ we first claim that $G$ is isomorphic to $K_{n-s}^n(l_1,\ldots,l_{n-s})$ for some $l_1,l_2,\ldots,l_{n-s}.$

By Lemma \ref{c} and Lemma \ref{c_0}, every block of $G$ is complete and every cut vertex of $G$ is shared by exactly two blocks. If $s=0$, then $G$ has exactly one block and $G=K_n$ also $K_n$  is isomorphic to $K_n^n(1,1,\cdots,1).$ 

For $s=1,$ $G$ has exactly two complete blocks with a common vertex $w$ (say). Let $B_1$ and $B_2$ be the  two blocks of $G.$  If any of $B_1$ or $B_2$ is $K_2$ then $G$ is isomorphic to $K_{n-1}^n(2,1,\ldots,1)$.  Otherwise, let $V(B_1)=\{u_1,u_2,\ldots,u_{m_1}=w\}$ and $V(B_2)=\{v_1,v_2,\ldots,v_{m_2}=w\}$ with $m_1,m_2 > 2.$ Construct a new graph $G'$ from $G$ as follow: Delete the edges $\{u_1,u_i\} ,i=2,3,\ldots,m_1-1$ and add the edges $\{u_i,v_j\},i=2,3,\ldots,m_1-1;j=1,2,\ldots,m_2-1.$ Clearly $G'\in \mathfrak{C_{n,s}}.$ Then the only type of distances which increase are $d(u_1,u_j), j=2,3,\ldots u_{m_1-1}$ and each such distance increases by one. So total increment in distance is exactly $m_1-2.$ Also each distance $d(u_i,v_j),\; i=2,3,\ldots,m_1-1; j=2,3,\ldots m_2-1$ decreases by one. The total decrement is $(m_1-2)(m_2-1).$ Since $m_1,m_2>2,$ so $W(G')<W(G)$, which is a contradiction. Hence $G$ is isomorphic to $ K_{n-1}^n(2,1,\ldots,1)$.

Now suppose $s\geq 2.$ Then $G$ has $s+1$ blocks and also $G$ has either one central block or two adjacent central blocks. 

\underline{Claim:} All non central blocks of $G$ are $K_2.$\\
Suppose $B$  is a non central block of $G$ which is not $K_2$. Then by Lemma \ref{c_3}, $B$ must be a non pendant block. Construct $G'$ from $G$ as in Lemma \ref{c_1}. Clearly $G'\in \mathfrak{C_{n,s}}$  and  by Lemma \ref{c_1}, $W(G')<W(G)$ which is a contradiction. 

If $G$ has exactly one central block, then  $G$ is  isomorphic to $K_{n-s}^n(l_1,\ldots,l_{n-s})$ for some $l_1,l_2,\ldots,l_s$. Suppose $G$ has two central blocks and $G$ is not isomorphic to $K_{n-s}^n(l_1,\ldots,l_{n-s})$ for any $l_1,l_2,\ldots,l_{n-s}$. Then each of the central blocks of $G$ has at least $3$ vertices. Let $B_1$ and $B_2$ be the two central blocks with a common vertex $w.$ Let  $V(B_1)=\{u_1,u_2,\ldots,u_{m_1}=w\}$ and $V(B_2)=\{v_1,v_2,\ldots,v_{m_2}=w\}$ with $m_1,m_2 > 2.$ Let $H_1$($H_2$) be the maximal connected component of $G$ containing exactly one vertex $w$ of $B_2$($B_1$). Let $P_l:wu_1t_3\cdots t_l$ be the longest path in $H_1$ starting at $w$ containing $u_1$ such that non of the vertex $t_3,\ldots,t_l$ belongs to $B_1.$ Take $w$ as $t_1$ and $u_1$ as $t_2$ in $P_l.$ Since $B_1$ and $B_2$ are central blocks, so there exists a path $P'_l:t_1't_2'\cdots t_l'$ on $l$ vertices in $H_2$ starting at $w=t_1'$ and containing exactly two vertices of $B_2.$ Construct a new graph $G'$ from $G$ as follow: Delete the edges $\{u_1,u_i\} ,i=2,3,\ldots,m_1-1$ and add the edges $\{u_i,v_j\},i=2,3,\ldots,m_1-1;j=1,2,\ldots,m_2-1.$ Clearly $G'\in \mathfrak{C_{n,s}}.$ The only type of distances which increase in $G'$ are $d_{G'}(u,t_j)$ where $u\in V(H_1)\setminus V(P_l)$ and $j=2,\ldots,l$ also each such distance increases by one. The  distance $d_{G'}(u,t_j')$ decreases by one where $u\in V(H_1)\setminus V(P_l)$ and $j=2,\ldots,l$. So, the increment in distance by the pairs $u,t_j$ are neutralized by the pairs $u,t_j'.$ Since $m_2\geq 3$, there exist at least one vertex $w'$ in $B_2$ which is not in $P_l'.$ For each $u\in V(H_1)\setminus V(P_l)$, the distance $d_{G'}(u,w')$ decreases by one. So, $W(G')< W(G)$, which is a contradiction. Hence $G$ is $K_{n-s}^n(l_1,\ldots,l_{n-s})$ for some $l_1,l_2,\ldots,l_{n-s}.$ Now the result follows from Lemma \ref{c_2}.
\end{proof}

\noindent{\bf Addresses}:\\

\noindent 1) School of Mathematical Sciences,\\
National Institute of Science Education and Research (NISER), Bhubaneswar,\\
P.O.- Jatni, District- Khurda, Odisha - 752050, India\medskip

\noindent 2) Homi Bhabha National Institute (HBNI),\\
Training School Complex, Anushakti Nagar,\\
Mumbai - 400094, India\medskip

\noindent E-mails: dinesh.pandey@niser.ac.in, klpatra@niser.ac.in

\end{document}